\documentclass[10pt, oneside]{amsart}

\usepackage[utf8]{inputenc}
\usepackage{amsmath}
\usepackage{amssymb}
\usepackage{amsfonts}
\usepackage{amsthm}
\usepackage{thmtools}
\usepackage{enumitem}
\usepackage[top=3cm, bottom=3cm, left=3cm, right=2.5cm]{geometry}
\usepackage{graphicx}
\usepackage{titlesec}
\usepackage{hyperref}
\usepackage{xcolor}
\usepackage{setspace}
\usepackage{verbatim}
\usepackage{tikz-cd}

\hypersetup{colorlinks=true, linkcolor=blue, citecolor = blue, filecolor = blue, urlcolor = blue}

\DeclareSymbolFontAlphabet{\amsmathbb}{AMSb}

\declaretheoremstyle[bodyfont=\normalfont]{normalbody}

\declaretheorem[numberwithin=section,name=Theorem]{theorem}
\declaretheorem[sibling=theorem,style=normalbody,name=Definition]{definition}
\declaretheorem[sibling=theorem,name=Corollary]{corollary}
\declaretheorem[sibling=theorem,name=Lemma]{lemma}
\declaretheorem[sibling=theorem,name=Proposition]{proposition}

\declaretheorem[sibling=theorem,style=normalbody,name=Example]{example}
\declaretheorem[sibling=theorem,style=normalbody,name=Examples]{examples}
\declaretheorem[sibling=theorem,style=normalbody,name=Remark]{remark}
\declaretheorem[sibling=theorem,style=normalbody,name=Remarks]{remarks}
\declaretheorem[sibling=theorem,style=normalbody,name=Question]{question}

\declaretheorem[numbered=no,name=Theorem]{theorem-no}
\declaretheorem[numbered=no,name=Corollary]{corollary-no}

\newcommand{\Z}{\mathbb{Z}}
\newcommand{\N}{\mathbb{N}}
\newcommand{\Q}{\mathbb{Q}}
\newcommand{\R}{\mathbb{R}}

\newcommand{\G}{\mathbb{G}}

\newcommand{\s}[1]{\mathcal{#1}}

\newcommand{\Hom}{\operatorname{Hom}}
\newcommand{\Ext}{\operatorname{Ext}}

\newcommand{\Br}{\operatorname{Br}}

\newcommand{\coker}{\operatorname{coker}}

\newcommand{\sm}{\operatorname{sm}}

\newcommand{\sHom}{\underline{\operatorname{Hom}}}

\newcommand{\Spec}{\operatorname{Spec}}
\newcommand{\id}{\operatorname{id}}

\newcommand{\Ab}{\mathbf{Ab}}

\newcommand{\im}{\mathrm{Im}\mkern1mu}

\newcommand{\cts}{\mathrm{cts}}
\newcommand{\LCA}{\mathbf{LCA}}
\newcommand{\Haus}{\mathrm{Haus}}
\newcommand{\TAb}{\mathbf{TAb}}

\titleformat{\section}[block]
		{\normalsize\sc}{\thesection}{1em}{\centering\normalsize}
\titleformat{\subsection}[block]
		{\normalsize\bf}{\thesubsection}{1em}{\centering\normalsize}

\tikzcdset{
	arrow style=math font,
}

\tikzset{
  symbol/.style={
    draw=none,
    every to/.append style={
      edge node={node [sloped, allow upside down, auto=false]{$#1$}}}
  }
}

\expandafter\def\expandafter\normalsize\expandafter{%
	\normalsize%
	\setlength\abovedisplayskip{.5em}%
	\setlength\belowdisplayskip{.5em}%
	\setlength\abovedisplayshortskip{.5em}%
	\setlength\belowdisplayshortskip{.5em}%
}

\title{Topologies on abelian groups and a topological five-lemma}

\author{Felipe Rivera-Mesas}
\address{Departamento de Matemáticas, Facultad de Ciencias, Universidad de Chile}
\email{felipe.rivera.m at ug.uchile.cl}
\thanks{ 
	\mbox{\hspace{1.3em}} {\sl MSC codes (2020):} 22B05. \\
	\mbox{\hspace{1.3em}} This work was partially supported by the Beca de Doctorado Nacional ANID folio 21210171, awarded in 2021.}

\begin{document}
	
\begin{abstract}
In this article we establish some results that allow to deduce the continuity of homomorphisms of (topological) abelian groups from commutative diagrams. In particular, we present a new topological version of the classical Five-Lemma. These results aim to be applied in duality results between cohomology groups in arithmetical contexts. In such a topological-arithmetical context, Pontryagin duality plays a central role and it becomes necessary to know whether certain homomorphisms are continuous.
\end{abstract}

\maketitle
	
\section{Introduction}

Given two abelian topological groups $A$ and $B$, a \emph{topological extension} of $B$ by $A$ is a topological group $G$ that fits in a exact sequence of groups
	\[ 0 \to A \overset{\iota}{\to} G \overset{\pi}{\to} B \to 0, \]
where $\iota$ is a topological embedding and $\pi$ is a quotient map. Let $\TAb$ be the category of abelian topological groups, whose objects are abelian topological groups and morphisms are continuous homomorphisms. We will refer an \emph{extension of abelian topological groups} (or an \emph{extension in $\TAb$}) as a topological extension of the abelian topological groups involved. Every extension in $\TAb$ induces, by forgetting the topology, an extension in $\Ab$. Conversely, an interesting question arises: 

\begin{question} \label{question}
	Given $A,B\in\TAb$ and the following extension in $\Ab$ 
	\begin{equation} \label{exact sequence}
		 0 \to A \overset{\iota}{\to} G \overset{\pi}{\to} B \to 0.
	\end{equation}
	\begin{center}
	\sl Can $G$ be topologized such that \eqref{exact sequence} becomes an extension in $\TAb$?
	\end{center} 
\end{question}

Every extension in $\Ab$, as \eqref{exact sequence}, is codified by a factor set $h_s:B\times B\to A$ induced by a set-theoretic section $s$ of $\pi$ satisfying $s(0_B)=0_G$ (see \cite[\S 6.6]{Weibel}). More precisely,
	\[ h_s(b,b')=s(b)+s(b')-s(b+b') \]
induces a group operation $+_s$ on $A\times B$ via
	\[ (a,b) +_s (a',b') := (a+a'+h_s(b,b'),b+b'). \]
Thus, $G$ is isomorphic to $(A\times B,+_s)$ via
	\begin{align*}
	\begin{split}
	\theta_s:(A\times B,+_s) &\to G \\
	(a,b) &\mapsto \iota(a)+s(b).
	\end{split}
	\end{align*}
Then, when $h_s$ is continuous at $(0_B,0_B)$, the group $(A\times B,+_s)$ becomes a topological group equipped with the topology $\tau$ whose basis of neighborhoods of $(0_A,0_B)$ is the product of open neighborhoods $U\times V$ of $0_A$ and $0_B$, respectively. Thus,
	\[ 0 \to A \to (A\times B,+_s,\tau) \to B \to 0 \]
becomes an extension in $\TAb$. Hence, $G$ equipped with the topology $\tau_s$, induced by $\theta_s$, answers affirmatively to the Question \ref{question}. Further, the section $s:B\to(G,\tau_s)$ becomes continuous at $0_B$. On the other hand, choosing a topology $\tau$ on $G$ such that $s:B\to(G,\tau)$ is continuous at $0_B$, we also answer affirmatively to Question \ref{question} (see \cite[\S 4]{Calabi}). For instance, when $B$ is discrete, there exists a unique topology on $G$ answering affirmatively to Question \ref{question} (see Corollaire 2 of Proposition 4.1 in \emph{loc. cit.}). \\

Hirosi Nagao proved that if $A,B\in\TAb$ are first countable, then every topological extension $G$ of $B$ by $A$ comes from the choice of a set-theoretical section $s:B\to G$, satisfying $s(0_B)=0_G$, continuous at $0_B$ (\cite[Theorem 2]{Nagao}). Thus, when $A$ and $B$ in Question \ref{question} are first countable, the problem is equivalent to finding a convenient topology on $G$ for which some set-theoretic section $s:B\to G$, satisfying $s(0_B)=0_G$, is continuous at $0_B$. \\

George Mackey proved that if $A,B\in\TAb$ are Hausdorff, locally compact and second countable, then every topological extension $G$ of $B$ by $A$ comes from the choice of a set-theoretical section $s:B\to G$, satisfying $s(0_B)=0_G$, that is borelian (\cite[Theorem 2 and 3]{Mackey57}). This result was extended by Calvin Moore to the case where $A$ is a polish group and he proved that $s$ may be borelian almost everywhere (see \cite[Theorem 10]{Moore}). Recall that $G$ is a polish group if its underlying topological space is a second countable complete metric space. Thus, when $A$ and $B$ in Question \ref{question} are as in Moore's, the problem is equivalent to find a convenient topology on $G$ for which some set-theoretic section $s:B\to G$, satisfying $s(0_B)=0_G$, is borelian almost everywhere.

\vspace{1em}

In this paper topological groups are {\it not} assumed to be Hausdorff. In particular, we will work with non-Hausdorff locally compact topological abelian groups. In this context, a topological group is \emph{locally compact} if the identity element has a neighborhood basis composed by compact neighborhoods. Recall that every abelian topological group $G$ has a universal Hausdorff quotient $G_{\Haus}:=G/\overline{\{0_G\}}$. Also recall that a continuous homomorphism $f:G\to H$ is \emph{strict} if the induced map $\tilde{f}:G/\ker f\to\im f$ is a topological isomorphism, where $G/\ker f$ has the quotient topology and $\im f$ has the subspace topology induced from $H$ (see Definition \ref{strict def}). \\

The following result is aimed to be applied to the problem of ensuring the continuity of homomorphisms between topological groups without explicitly verifying this property.

\begin{theorem-no}[Proposition \ref{Five-Lemma nagao}]
	Let
	\begin{equation*}
		\begin{tikzcd}
			\s{E}_1: 0 \ar{r} & A_1 \ar{r}{\iota_1} \ar{d}{\alpha} & G_1 \ar{r}{\pi_1} \ar{d}{\gamma} & B_1 \ar{r} \ar{d}{\beta} & 0 \\
			\s{E}_2: 0 \ar{r} & A_2 \ar{r}{\iota_2} & G_2 \ar{r}{\pi_2} & B_2 \ar{r} & 0
		\end{tikzcd}
	\end{equation*}
	be a commutative diagram of abelian groups. Assume that the rows of the above diagram are topological extensions of first countable and locally compact abelian topological groups. Then, in each of the following cases:
	\begin{enumerate}[nosep, label=\alph*)]
		\item $B_1$ is discrete or
		\item $A_2$ is Hausdorff compact and $\s{E}_1$ satisfies some of the following conditions
			\begin{enumerate}[nosep, label=\roman*)]
				\item $A_1$ is Hausdorff compact or
				\item $B_1$ is Hausdorff and $A_1$ is second countable,
			\end{enumerate}
	\end{enumerate}
	the map $\gamma_{\Haus}$ is well-defined and continuous (resp. continuous and strict) whenever $\alpha$ and $\beta$ are continuous (resp. continuous and strict). In particular, if we further assume that $G_2$ is Hausdorff, then $\gamma$ is continuous (resp. continuous and strict) if, and only if, $\alpha$ and $\beta$ are continuous (resp. continuous and strict).
\end{theorem-no}

As a consequence of the Theorem above, we show a topological version of the well-known Five-Lemma in homological algebra. In the literature there do exist topological versions of this lemma, for instance, in the category $\LCA$ of Hausdorff locally compact abelian topological groups (see Proposition 2.8 in \cite{Fulp}) or in the category of Banach spaces (see Theorem of Section 3 in \cite{Pryde}). The results just mentioned make some topological assumptions on the maps that appear in the Five-Lemma diagram. In particular, both results are stated assuming that all maps are continuous homomorphisms. The result that we present below differs from these results since for us the middle vertical arrow is just a homomorphism of abelian groups. In particular, the continuity of the vertical middle arrow is not defined a priori.

\begin{corollary-no}[Topological Five-Lemma]
	Let
	\begin{equation*} 
		\begin{tikzcd}
			A_1 \ar{r}{f_1} \ar{d}{\alpha} & B_1 \ar{r}{g_1} \ar{d}{\beta} & C_1 \ar{r}{h_1} \ar{d}{\gamma} & D_1 \ar{r}{k_1} \ar{d}{\delta} & E_1 \ar{d}{\epsilon} \\
			A_2 \ar{r}{f_2} & B_2 \ar{r}{g_2} & C_2 \ar{r}{h_2} & D_2 \ar{r}{k_2} & E_2
		\end{tikzcd}
	\end{equation*}
	be a commutative diagram of abelian groups, where the rows are strict exact sequences (see Definition \ref{strict sequence}) of locally compact abelian topological groups. Suppose that
	\begin{enumerate}[nosep,label=\roman*)]
		\item $B_i$ and $D_i$ are first countable,
		\item $\beta$ and $\delta$ are topological isomorphisms,
		\item $\epsilon$ is injective and
		\item $\alpha$ is surjective.
	\end{enumerate}
	Then, in each of the following cases
	\begin{enumerate}[nosep, label=\alph*)]
		\item $D_i$ is discrete; or
		\item $B_i$ is Hausdorff and compact,
	\end{enumerate}
	$\gamma_{\Haus}$ is well-defined and it is a quotient map. Whence $\gamma$ is a topological isomorphism when $C_2$ is Hausdorff.
\end{corollary-no}

Our topological version of the Five-Lemma can be applied to prove the continuity of certain Yoneda pairings between (hyper)cohomology groups of complexes of \'etale abelian sheaves in the context of local Tate duality results (cf. for instance \cite[Theorem 2.3]{HS05}). \\

\begin{comment}
Let us finish this introduction with the following comments. In Section \ref{nagao section} of this article we work only under the condition that all topological groups involved are first countable locally compact but not necessarily Hausdorff. The reason for this is that we need to apply our results in a context where non-Hausdorff groups occur naturally. However, note that in Theorem \ref{c1} we imposed on the groups involved the hypothesis of being Hausdorff. We do this because, once we can ensure that certain maps between non-Hausdorff groups are continuous, we are able to replace the non-Hausdorff groups involved by their profinite completions, which are always Hausdorff and compact. Remark \ref{duality} at the end of Section \ref{proof} briefly explains this methodology.
\end{comment}

%%%%%%%%%%%%%%%%%%%%%%%%%%%%%%%%%%%%%%%%%%%%%%%%%%%%%%%%%%%%%%%%%%%%%%%%%%%%%%%%%%%%%%%%
%%%%%%%%%%%%%%%%%%%%%%%%%%%%%%%%%%%%%%%%%%%%%%%%%%%%%%%%%%%%%%%%%%%%%%%%%%%%%%%%%%%%%%%%
%%%%%%%%%%%%%%%%%%%%%%%%%%%%%%%%%%%%%%%%%%%%%%%%%%%%%%%%%%%%%%%%%%%%%%%%%%%%%%%%%%%%%%%%
%%%%%%%%%%%%%%%%%%%%%%%%%%%%%%%%%%%%%%%%%%%%%%%%%%%%%%%%%%%%%%%%%%%%%%%%%%%%%%%%%%%%%%%%

\section*{Acknowledgments}

I would like to thank my advisors, Cristian Gonz\'alez Avil\'es and Giancarlo Lucchini Arteche, for their constant support while writing this article. I also thank {\it Agencia Nacional de Investigaci\'on y Desarrollo} (ANID) for funding my research via Beca de Doctorado Nacional ANID folio 21210171. 

%%%%%%%%%%%%%%%%%%%%%%%%%%%%%%%%%%%%%%%%%%%%%%%%%%%%%%%%%%%%%%%%%%%%%%%%%%%%%%%%%%%%%%%%
%%%%%%%%%%%%%%%%%%%%%%%%%%%%%%%%%%%%%%%%%%%%%%%%%%%%%%%%%%%%%%%%%%%%%%%%%%%%%%%%%%%%%%%%
%%%%%%%%%%%%%%%%%%%%%%%%%%%%%%%%%%%%%%%%%%%%%%%%%%%%%%%%%%%%%%%%%%%%%%%%%%%%%%%%%%%%%%%%
%%%%%%%%%%%%%%%%%%%%%%%%%%%%%%%%%%%%%%%%%%%%%%%%%%%%%%%%%%%%%%%%%%%%%%%%%%%%%%%%%%%%%%%%

\section{Basics on topological groups}

We start by recalling some definitions.

\begin{definition} \label{strict def}
	A continuous homomorphism of abelian topological groups $f:A\to B$ is called \emph{strict} if the image of every open subset is open in the image $f(A)$ with the subspace topology in $B$, or equivalently (see \cite[III, \S 2.8, p. 236, Proposition 24]{GTBou}), if the map $\tilde{f}:A/\ker f\to f(A)$ is an isomorphism of topological groups, where $A/\ker f$ has the quotient topology.
\end{definition}

\begin{examples} \label{strict}
	Let $f:A\to B$ be a continuous homomorphism of abelian topological groups. Then $f$ is strict in the following cases:
	\begin{enumerate}[nosep, label=\alph*)] 
		\item $A$ is an open subgroup of $B$ and $f$ is the canonical inclusion. However, $f$ can be non-strict if $A$ is not open, e.g., when $A=\Z$ is equipped with the discrete topology and $B=\Z_p$ is equipped with $p$-adic topology. Indeed, $\Z$ equipped with the $p$-adic topology is dense in $\Z_p$ and therefore cannot be discrete since $\Z_p$ is Hausdorff. More generally, $f$ is strict whenever $f$ is open.
		\item $B$ has the discrete topology;
		\item $A$ is compact and $B$ is Hausdorff (see \cite[III, \S 2.8, p. 237, Remark 1]{GTBou});
		\item $A$ is locally compact and second countable, $B$ is Hausdorff and locally compact and $\im f$ is closed in $B$ (see \cite[p. 42, Theorem 5.29]{Hewitt}).
	\end{enumerate}
\end{examples}

\begin{remark} \label{closed image}
	Let $f:A\to B$ be a homomorphism of topological groups, where $A$ is locally compact and $B$ is Hausdorff. If $f$ is strict, then $\im f$ is a closed subgroup of $B$. Indeed, since $A/\ker f$ is locally compact, $\im f$ is a locally compact subgroup of $B$, which is necessarily closed by \cite[p. 45, Corollary 4.7]{Stroppel}. 
\end{remark}

\begin{definition} \label{strict sequence}
	An exact sequence of abelian topological groups
	\[ \cdots \to A_{i-1} \overset{f_{i-1}}{\to} A_i \overset{f_i}{\to} A_{i+1} \to \cdots \]
	is called \emph{strict exact} if every map $f_i$ is strict. A short strict exact sequence of abelian topological groups
	\[ 0 \to A \to B \to C \to 0 \]
	is called a \emph{topological extension}.
\end{definition}

\begin{remark} \label{extension properties}
	Given a topological extension $0\to A \to B \to C \to 0$ of abelian topological groups, we can view $A$ as a subgroup of $B$ with the subspace topology and therefore $C$ is homeomorphic to $B/A$ equipped with the quotient topology. Note that $B/A$ is Hausdorff if and only if $A$ is a closed subgroup of $B$.
\end{remark}

There are some topological properties that are preserved by topological extension, that is, given a topological extension of abelian topological groups 
	\[ 0\to A\to B\to C \to 0, \]
where $A$ and $C$ satisfy such a property, then $B$ also satisfies it. These properties are called \emph{extension properties}. For example, the property of being
\begin{enumerate}[nosep, label=(\arabic*)]
	\item Hausdorff;
	\item first countable;
	\item profinite;
	\item locally compact;
	\item discrete; or
	\item totally disconnected,
\end{enumerate}
is an extension property. For $(2)$ see \cite[p 47, 5.38.(e)]{Hewitt}; for $(3)$-$(6)$ see \cite[p. 58, Theorem 6.15]{Stroppel}\footnote{In \cite{Stroppel}, strict morphisms are called \emph{proper morphisms}.}. For (1) we have to prove that $\{0\}$ is closed in $B$. Since $C$ is Hausdorff and the projection $B\to C$ is continuous, $A$ is a closed subgroup of $B$ and the closure of $\{0\}$ in $B$ is contained in $A$. Hence $\{0\}$ is closed in $B$ since $A$ is Hausdorff.

%%%%%%%%%%%%%%%%%%%%%%%%%%%%%%%%%%%%%%%%%%%%%%%%%%%%%%%%%%%%%%%%%%%%%%%%%%%%%%%%%%%%%%%%
%%%%%%%%%%%%%%%%%%%%%%%%%%%%%%%%%%%%%%%%%%%%%%%%%%%%%%%%%%%%%%%%%%%%%%%%%%%%%%%%%%%%%%%%

\subsection*{Pontryagin dual}

Let $A$ and $B$ be two abelian topological groups. We denote by $\Hom_{\cts}(A,B)$ the group of continuous homomorphisms from $A$ to $B$. The group $\Hom_{\cts}(A,B)$ can be endowed with a natural topology, as follows: given a compact subset $K$ of $A$ and an open subset $U$ of $B$, we define $V(K,U)$ as
\[ V(K,U) := \{f:A\to B\text{ continuous}\mid f(K)\subseteq U\}. \]
The family $\{V(K,U)\}_{K,U}$ generates a topology on the group of all the continuous maps $C(A,B)$ from $A$ to $B$ which is called the \emph{compact-open topology on} $C(A,B)$. Then $\Hom_{\cts}(A,B)$ is endowed with the subspace topology inherehited from $C(A,B)$.

\begin{remark}
	When $B$ is Hausdorff, the topological group $\Hom_{\cts}(A,B)$ is Hausdorff, even if $A$ is not (see \cite[p. 91, Lemma 9.2.(c)]{Stroppel}).
\end{remark}

Let $\mathbb{T}$ be the unit circle $\R/\Z$ with the quotient topology induced by $\R$ with its usual metric topology. The torsion subgroup of $\mathbb{T}$ is $\Q/\Z$ and has discrete topology (see \cite[p. 31, Corollary 3.13]{Stroppel}).

\begin{definition}
	For $A$ an abelian topological group, we define \emph{the Pontyagin dual of $A$} as the abelian topological group $A^*:=\Hom_{\cts}(A,\mathbb{T})$.
\end{definition}

\begin{remark}
	Since $\mathbb{T}$ is Hausdorff, $A^*$ is Hausdorff for any (possibly non-Hausdorff) $A$.
\end{remark}

Let $\LCA$ be the category of locally compact Hausdorff abelian topological groups whose morphisms are the continuous group homomorphisms. The asignment $A\to A^*$ defines a contravariant funtor $(-)^*:\LCA\to\LCA$. On morphisms, this functor acts by right-composition, that is, it sends a continuous homomorphism $f:G\to H$ to the morphism $f^*:H^*\to G^*$ defined by $f^*(g:H\to\mathbb{T}):=g\circ f$. Observe that this definition does not depend on a topological property of the morphisms themselves. Furthermore, this functor satisfies the following  important and well-known result.

\begin{proposition}[Pontryagin duality theorem] \label{Pontryagin duality theorem}
	The functor $(-)^*\circ(-)^*:\LCA\to\LCA$ is naturally isomorphic to the identity functor $\id_{\LCA}$. Morever, if $\s{E}:0 \to A \to B \to C \to 0$ is a topological extension in $\LCA$, the dual sequence
	\[ \s{E}^*:0\to C^* \to B^* \to A^* \to 0 \]
	is also a topological extension. Furthermore, the functor $(-)^*:\LCA\to\LCA$ establishes an anti-equivalence between the category of discrete and torsion abelian groups and the category of profinite abelian groups.
\end{proposition}

\begin{proof}
	For the first assertion, see \cite[p. 193, Theorem 22.6]{Stroppel}; for the second, see \cite[p. 195, Corollary 23.4]{Stroppel}. For the last, see \cite[p. 62, Theorem 2.9.6]{profinite}.
\end{proof}

\begin{remark} \label{continuity pontryagin}
	Let $f:G\to H$ be a homomorphism of abelian groups, where $G$ and $H$ are objects in $\LCA$. Then $f$ is continuous if, and only if, $f^*$ is continuous.
\end{remark}

%%%%%%%%%%%%%%%%%%%%%%%%%%%%%%%%%%%%%%%%%%%%%%%%%%%%%%%%%%%%%%%%%%%%%%%%%%%%%%%%%%%%%%%%%%%%%%%%%%%%
%%%%%%%%%%%%%%%%%%%%%%%%%%%%%%%%%%%%%%%%%%%%%%%%%%%%%%%%%%%%%%%%%%%%%%%%%%%%%%%%%%%%%%%%%%%%%%%%%%%%

\subsection*{Separation functor}

The importance of the hypothesis of being Hausdorff in the Pontryagin duality lies on the following simple fact: for every continuous homomorphism $f:A\to B$, where $B$ is Hausdorff, we have $\overline{\{0_A\}}\subseteq\ker f$. Thus, if the evaluation map $A\to A^{**}$ (which is the map that defines the natural isomorphism in the Pontryagin duality) is an isomorphism of topological groups, then $\overline{\{0_A\}}$ is trivial and, therefore, $A$ is Hausdorff. We may, however, treat the non-Hausdorff groups by using the following well-known construction. 

\begin{definition} \label{haus def}
	For an abelian topological group $G$, we define \emph{the separation of $G$} as the quotient $G_{\Haus}:=G/\overline{\{0_G\}}$. We denote by $q_G:G\to G_{\Haus}$ the canonical projection.
\end{definition}

Let $G$ be an abelian topological group. The separation $G_{\Haus}$ is universal for continuous homomorphisms from $G$ to Hausdorff abelian topological groups. That is, every continuous homomorphism $f:G\to H$ with $H$ Hausdorff factors through $q_G$. Furthermore, the assignment $G\mapsto G_{\Haus}$ defines a functor from the category of abelian topological groups to the category of Hausdorff abelian topological groups and has the following properties:
\begin{enumerate}[nosep, label=\roman*.]
	\item if $f:G\to H$ is a continuous and strict homomorphism, then $f_{\Haus}:G_{\Haus}\to H_{\Haus}$ is continuous and strict (see \cite[Lemma 1.9.(i)]{GA24});
	\item if $f:G\to H$ is a continuous and surjective homomorphism, then $f_{\Haus}:G_{\Haus}\to H_{\Haus}$ is continuous and surjective; and
	\item the projection $q_G:G\to G_{\Haus}$ induces canonical isomorphisms $q_G^*:(G_{\Haus})^*\to G^*$ and $q_G^\wedge:G^\wedge\to (G_{\Haus})^\wedge$, where $(-)^\wedge$ denotes the profinite completion functor. To see that $q_G^\wedge$ is an isomorphism, see Lemma 1.9.(ii) in \emph{loc. cit.}
\end{enumerate}
In particular, if $G$ is locally compact, then $(G^*)^*$ is canonically isomorphic to $G_{\Haus}$. Furthermore, for every continuous homomorphism $f:G\to H$, where $G$ and $H$ are locally compact, the Pontryagin duality functor sends $f_{\Haus}$ to $f^{**}$. Thus, $f_{\Haus}$ is continuous if, and only if, $f^*$ is continuous. Unfortunately, the functor $(-)_{\Haus}$ is not exact and care is needed when discussing the Pontryagin duality theorem in a context where non-Hausdorff groups intervene. At any rate, under suitable hypotheses on a topological extension $\s{E}:0\to A\to G\to B\to 0$, we can ensure that
	\[ \s{E}_{\Haus}: 0\to A_{\Haus}\to G_{\Haus} \to B_{\Haus} \to 0 \]
is a topological extension. To do that, we will use the following lemma.

\begin{lemma} \label{strictness and injectivity}
	Let the following commutative diagram in $\TAb$
	\begin{equation*}
		\begin{tikzcd}
			A \ar[hook]{r}{f} \ar[hook]{d}[swap]{\alpha} & B \ar[hook]{d}{\beta} \\ 
			A' \ar[hook]{r}{g} & B'
		\end{tikzcd}
	\end{equation*}
	where all maps are injective. Suppose that $f,g$ and $\beta$ are strict. Then $\alpha$ is strict.
\end{lemma}

\begin{proof}
	Since $g$ is injective, the following equality of sets holds
	\[ \alpha(X)=g^{-1}(\beta(f(X))) \]
	for every subset $X\subseteq A$. Let $U\subseteq A$ an open subset.	Since $f$ is strict, there exists an open subset $V\subseteq B$ such that $f(U)=V\cap f(A)$. Then,
	\[ \alpha(U)=g^{-1}(\beta(V\cap f(A))) \overset{(\ast)}{=} g^{-1}(\beta(V)) \cap g^{-1}(\beta(f(A))) = g^{-1}(\beta(V)) \cap \alpha(A), \]
	where equality $(\ast)$ follows from the injectivity of $\beta$. Now, since $\beta$ is strict, there exists an open subset $W\subseteq B'$ such that $\beta(V)=\beta(B)\cap W$. Then,
	\[ \alpha(U) = g^{-1}(\beta(B)\cap W) \cap \alpha(A) = g^{-1}(\beta(B))\cap g^{-1}(W) \cap \alpha(A). \]
	Observe that $g^{-1}(\beta(B))$ contains $\alpha(A)$. Thus,
	\[ \alpha(U) = g^{-1}(\beta(B)\cap W) \cap \alpha(A) = g^{-1}(W) \cap \alpha(A). \]
	Hence, $\alpha(U)$ is open in $\alpha(A)$ since $g^{-1}(W)$ is an open subset of $A'$.
\end{proof}

The following lemma give us some conditions under which the separation functor $(-)_{\Haus}$ preserves topological extensions.

\begin{lemma} \label{haus exactness}
	Let $\s{E}:0\to A\overset{\iota}{\to} G\overset{\pi}{\to} B\to 0$ be a topological extension of locally compact groups. Then, in each of the following cases:
	\begin{enumerate}[nosep, label=\alph*)]
		\item $A$ is Hausdorff compact or
		\item $B$ is Hausdorff and $A$ is second countable,
	\end{enumerate}
	the topological extension $\s{E}$ yields the following extension in $\TAb$
	\[ \s{E}_{\Haus}:0\to A_{\Haus} \overset{\iota_{\Haus}}{\to} G_{\Haus}\overset{\pi_{\Haus}}{\to} B_{\Haus}\to 0. \]
	In particular, in the both cases above, the sequence
	\[ \s{E}^*:0\to B^*\overset{\pi^*}{\to} G^*\overset{\iota^*}{\to} A^*\to 0 \]
	is as extension in $\LCA$.
\end{lemma}

\begin{proof}
	We have the following commutative diagram of abelian topological groups whose rows are topological extensions
	\begin{equation*} 
		\begin{tikzcd}
			\s{E}: & 0 \ar{r} & A \ar{r}{\iota} \ar{d}{f} & G \ar{r}{\pi} \ar{d}{q_G} & B \ar{r} \ar{d}{q_B} & 0 \\
			\s{E}_{\Haus}: & 0 \ar{r} & \ker\pi_{\Haus} \ar{r} & G_{\Haus} \ar{r}{\pi_{\Haus}} & B_{\Haus} \ar{r} & 0.
		\end{tikzcd}	
	\end{equation*}
	By the Snake Lemma, we have the following exact sequence of abelian topological groups
	\begin{equation} \label{haus sequence}
		0 \to \ker f \overset{\alpha}{\to} \overline{\{0_G\}} \overset{\pi'}{\to} \overline{\{0_B\}} \overset{\beta}{\to} \coker f \to 0.
	\end{equation}
	where $\alpha$ and $\pi'$ are continuous since $\iota$ and $\pi$ are so, respectively. Moreover, by Lemma \ref{strictness and injectivity}, $\alpha$ is strict since $\iota$ is so. Furthermore, by \cite[Proposition 4]{topsnake}, we know that $\beta$ is also continuous (but possibly not strict). Then:
	\begin{enumerate}[nosep,label=\alph*)]
		\item When $A$ is Hausdorff and compact, we have to prove that $f$ is a topological isomorphism. Note that $\ker f$ is Hausdorff since it is a subspace of $A$. Moreover, since $\alpha$ is continuous and strict, $\ker f$ is topologically isomorphic to subgroup of $\overline{\{0_G\}}$. However, $\overline{\{0_G\}}$ does not contain Hausdorff subgroups since every open subset of $\overline{\{0_G\}}$ contains $0_G$. Indeed, given any point $x\in\overline{\{0_G\}}$, every open neighborhood $U$ of $x$ contains $0_G$ and, therefore, $x$ and $0_G$ cannot be separated in disjoint open subsets of $\overline{\{0_G\}}$. Hence, $\ker f$ is trivial. Now, since $f$ is continuous, $A$ is compact and $\ker\pi_{\Haus}$ is Hausdorff, $f$ is strict (see Example \ref{strict}.c). Moreover, $\im f$ is a closed subgroup of $\ker\pi_{\Haus}$ (see Example \ref{closed image}). Hence, $\coker f$ is Hausdorff (see \cite[p. 45, Corollary 4.7]{Stroppel}). Hence, since $\beta$ is continuous, $\beta^{-1}(\{0_{\coker f}\})$ is a closed subset of $\overline{\{0_B\}}$ which contains $0_B$ and, therefore,
		\[ \overline{\{0_B\}} = \beta^{-1}(\{0_{\coker f}\}). \]
		Thus $\beta$ is the zero map and, consequently, $\coker f$ is trivial. Consequently, $f$ is a bijective, continuous and strict homomorphism. Hence $f$ is a topological isomorphism.
		\item When $B$ is Hausdorff, we have to prove that $A_{\Haus}$ is topologically isomorphic to $\ker\pi_{\Haus}$. By the sequence \eqref{haus sequence}, we know that $\coker f$ is trivial and $\ker f$ is isomorphic to $\overline{\{0_G\}}$ as topological groups. Then $f$ is strict since it is a surjective continuous map from a second countable locally compact space to a Hausdorff locally compact space (see Example \ref{strict}.d)). Now, since $B$ is Hausdorff, $\iota(A)$ is a closed subgroup of $G$. Consequently, $\iota$ induces an isomorphism of topological groups between $\overline{\{0_A\}}$ and $\overline{\{0_G\}}$. Thus $f$ is a surjective, continuous and strict homomorphism whose kernel is isomorphic to $\overline{\{0_A\}}$. We conclude that $\ker\pi_{\Haus}$ is isomorphic to $A_{\Haus}$ as topological groups.
	\end{enumerate}  
	The last assertion follows from the exactness of the Pontryagin dual on $\LCA$ and the fact that $q_G^*:(G_{\Haus})^*\to G^*$ is an isomorphism (see the paragraph below Definition \ref{haus def}).
\end{proof}

\begin{remark}
	In the case $a)$ of previous lemma, we may remove the assumption that $G$ and $B$ are locally compact to obtain a topological extension
	\[ \s{E}_{\Haus}:0\to A_{\Haus} \overset{\iota_{\Haus}}{\to} G_{\Haus}\overset{\pi_{\Haus}}{\to} B_{\Haus}\to 0. \]
\end{remark}

\begin{comment}
\begin{remark} \label{strict haus}
	Let $f:G\to H$ be a continuous homomorphism of abelian topological groups. When $H$ is Hausdorff, $f$ is strict if and only if $f_{\Haus}$ is strict. Indeed, $q_G$ is an open map and, when $H$ is Hausdorff, the equality $f=f_{\Haus}\circ q_G$ holds.
\end{remark}
\end{comment}

%%%%%%%%%%%%%%%%%%%%%%%%%%%%%%%%%%%%%%%%%%%%%%%%%%%%%%%%%%%%%%%%%%%%%%%%%%%%%%%%%%%%%%%%
%%%%%%%%%%%%%%%%%%%%%%%%%%%%%%%%%%%%%%%%%%%%%%%%%%%%%%%%%%%%%%%%%%%%%%%%%%%%%%%%%%%%%%%%
%%%%%%%%%%%%%%%%%%%%%%%%%%%%%%%%%%%%%%%%%%%%%%%%%%%%%%%%%%%%%%%%%%%%%%%%%%%%%%%%%%%%%%%%
%%%%%%%%%%%%%%%%%%%%%%%%%%%%%%%%%%%%%%%%%%%%%%%%%%%%%%%%%%%%%%%%%%%%%%%%%%%%%%%%%%%%%%%%

\section{Topological extensions} 

Consider the following extension in $\TAb$
\begin{equation} \label{top extension}
	\s{E}: 0 \to A \overset{\iota}{\to} G \overset{\pi}{\to} B \to 0,
\end{equation}
that is, all maps are continuous and strict homomorphisms. The underlying extension in $\Ab$
\begin{equation*}
	\s{|E|}: 0 \to A \overset{\iota}{\to} G \overset{\pi}{\to} B \to 0
\end{equation*}
is completely characterized by a (possibly non-unique) \emph{factor set} $h_s:B\times B\to A$ (see \cite[Definition 6.6.4]{Weibel}) defined by choosing a set-theoretic section $s:B\to G$ of $\pi$ via
	\[ h_s(b,b') = s(b) + s(b') - s(b+b'). \]
More precisely, $h_s$ induces a group operation $+_s$ on $A\times B$ via
\begin{equation} \label{operation}
	(a,b) +_s (a',b') := (a+a'+h_s(b,b'),b+b').
\end{equation} 
Thus, $G$ is algebraically isomorphic to $(A\times B,+_s)$ via
\begin{align} \label{theta}
	\begin{split}
		\theta_s:(A\times B,+_s) &\to G \\
		(a,b) &\mapsto \iota(a)+s(b).
	\end{split}
\end{align}
On the other hand, there exists some natural topologies on $A\times B$ to consider. For example, the product topology $\tau_{A\times B}$ or the topology $\tau_{A,B}$ whose basis of neighborhoods of $(0_A,0_B)$ is the product of open neighborhoods $U\times V$ of $0_A$ and $0_B$, respectively (see \cite[p. 33, Theorem 3.22]{Stroppel}). When $h_s$ is continuous (resp. continuous at $(0_B,0_B)$), the triple $(A\times B,+_s,\tau_{A\times B})$ (resp. $(A\times B,+_s,\tau_{A,B})$) forms an abelian topological group and
\begin{equation*}
	\s{E}: 0 \to A \overset{i}{\to} (A\times B,+_s) \overset{p}{\to} B \to 0,
\end{equation*}
becomes a topological extension, where $i$ and $p$ are the canonical inclusion and projection respectively. However, without assuming any continuity conditions on $h_s$, neither $\tau_{A\times B}$ nor $\tau_{A,B}$ may be compatible with $+_s$. Note that a continuity condition on $s$ implies a certain continuity condition on $h_s$ but not conversely in general. Furthermore, even when $(A\times B,+_s,\tau_{A,B})$ is a topological group, it is not always true that $\theta_s$ is an topological isomorphism. Nevertheless, there are some cases where $h_s$ and $+_s$ are compatible and $\theta_s$ is an isomorphism. One of these results is due to Nagao, who gives a characterization of the topological extensions between first countable topological groups.

\begin{proposition}[Theorem 4 \cite{Nagao}] \label{nagao theorem}
	Suppose that $A$ and $B$ are first countable. For every topological extension $G$ as in \eqref{top extension}, there exists a set-theoretic section $s:B\to G$ of $\pi$, satisfying $s(0_B)=0_G$, such that it is continuous at $0_B$, $(A\times B,+_s,\tau_{A,B})$ is an abelian topological group and $\theta_s$ is a topological isomorphism. Furthermore, a sequence $(\iota(a_i)+s(b_i))_{i\in\N}$ converges to $0_G$ if, and only if, $(a_i)_{i\in\N}$ and $(b_i)_{i\in\N}$ converge to $0_A$ and $0_B$, respectively.
\end{proposition}

This results also says us that any choice of a set-theoretic section $s:B\to G$ of $\pi$, satisfying $s(0_B)=0_G$, whose induced cocycle $h_s$ is continuous at $(0_B,0_B)$, induces a topology on $G$ compatible with its operation (see \cite[Proposition 4.1]{Calabi}). Further, all of these topologies are \emph{a priori} different. Nagao gave a criterion to know when two such sections $s$ and $s'$ define the same topology.

\begin{proposition}[Theorem 5 \cite{Nagao}] \label{nagao comparison}
	Let $A$ and $B$ be first countable abelian topological groups and let $G$ be the algebraic extension
		\[ 0 \to A \to G \overset{\pi}{\to} B \to 0. \]
	Two set-theoretic sections $s_i:B\to G$ of $\pi$ ($i=1,2$), satisfying $s_i(0_B)=0_G$, such that $h_{s_i}$ are continuous at $(0_B,0_B)$, define the same topology on $G$ if and only if the map
	\begin{align*}
	f_{s_1,s_2}:B &\to A \\
	b &\mapsto \iota^{-1}(s_1(b)-s_2(b))
	\end{align*}
	is continuous at $0_B$.
\end{proposition}

Another result is due to Mackey, which was improved by Moore, and works under the assumption of second countability on the groups involved. Recall that a \emph{polish group} is a separated topological group that admits a complete metric (see \cite[Proposition 1]{Moore}). For example, every second countable Hausdorff locally compact topological group is polish.
	
\begin{proposition}[Theorem 10 \cite{Moore}]
	 Suppose that $B$ is Hausdorff, locally compact and second countable and $A$ is a polish group. For every topological extension $G$ as in \eqref{top extension}, there exists a set-theoretic section $s:B\to G$ of $\pi$, satisfying $s(0_B)=0_G$, such that $s$ is borelian, $(A\times B,+_s,\tau_{A,B})$ is an abelian topological group and $\theta_s$ is a topological isomorphism. Furthermore, for every algebraic extension 
	 \begin{equation} \label{alg extension}
	 	0 \to A \to G \overset{\pi}{\to} B \to 0
	 \end{equation}
	 and a set-theoretic section $s:B\to G$ of $\pi$, satisfying $s(0_B)=0_G$, such that $h_s$ is borelian, there exists a unique topology on $G$ such that \eqref{alg extension} becomes a topological extension.
\end{proposition}

\begin{definition} \label{nagao top}
	If $(A\times B,+_s,\tau_{A,B})$ is a topological group, we will denote $\tau_{\s{E},s}$ the topology on $G$ induced by $\tau_{A,B}$ via $\theta_s$. We say that $s$ is a \emph{topologizing section}.
\end{definition}

\begin{remark}
	It seems that is not always true that for a abelian algebraic extension of two abelian topological groups, there exists a topologizing section. Observe that this problem differs from the problem of existence of topological extensions between two given abelian topological groups. The first problem asks whether a fixed algebraic extension can be topologized.
\end{remark}

It follows from the definition that, if $s$ is a topologizing section of $\pi$, then the following commutative diagram of abelian groups 
\begin{equation} \label{nagao diagram}
	\begin{tikzcd}
		0 \ar{r} & A \ar{r}{i} \ar[d,equal] & (A\times B,+_s) \ar{r}{p} \ar{d}[swap]{\theta_s} & B \ar{r} \ar[d,equal] & 0 \\
		0 \ar{r} & A \ar{r}{\iota} & G \ar{r}{\pi} & B \ar{r} \arrow[l, bend right=40, "s" swap] & 0
	\end{tikzcd}
\end{equation}
has strict exact rows and $\theta_s$ is a topological isomorphism and the section $s:B\to G$ is continuous at $0_B$.

\begin{remark}
	In general, the topology $\tau_{A,B}$ on $A\times B$ is coarser than the product topology $\tau_{A\times B}$, which implies that the product map $(f,g):X\to(A\times B,\tau_{A,B})$ is continuous for every topological space $X$ and every pair of continuous maps $f:X\to A$ and $g:X\to B$. To show that $\tau_{A,B}$ is strictly coarser than $\tau_{A\times B}$, we recall that a basis $\s{B}$ for the topology $\tau_{A,B}$ is the set of all translations of all products of open neighborhoods of $0$ (see \cite[p. 33, Theorem 3.22]{Stroppel}). Thus, due to the twisted group operation $+_s$, it can happen that a product $U\times V$ of open subsets $U\subseteq A$ and $V\subseteq B$ cannot be covered by open subsets in $\s{B}$. For example, assume that $A$ has discrete topology and $V\subseteq B$ is open and does not contain $0_B$. In order to have $\{a\}\times V$ in $\tau_{A,B}$ it is necessary that, for every $b\in V$, there exists an open neighborhood $W$ of $0_B$ such that 
	\[ (a,b)\in (a',b')+(\{0_A\}\times W) \subseteq \{a\}\times V, \]
	for some $(a',b')\in A\times B$. So if $W\neq\{0_B\}$ (i.e., $B$ does not have discrete topology), then we see that
	\[ (a'+h_s(b',w),b'+w) \in (a',b')+(\{0_A\}\times W), \]
	for some $w\in W$ different from $0_B$. However $a'+h_s(b',w)$ can differ from $a$ and therefore $\{a\}\times V$ does not belong to $\tau$. 
\end{remark}

\begin{comment}
We will denote by $\tau_{\s{E},s}$ the topology in $G$ induced by $\phi_s$. Note that it depends on the extension $\s{E}$ and the section $s$. Then $(G,\tau_{\s{E},s})$ is a first countable abelian topological group and the sequence \eqref{nagao} becomes a topological extension (see \cite[Theorem 3]{Nagao}). We have the following commutative diagram of abelian topological groups with strict exact rows
	\begin{equation} \label{nagao diagram}
		\begin{tikzcd}
			0 \ar{r} & A \ar{r}{i} \ar[d,equal] & (A\times B,+_s) \ar{r}{p} \ar{d}[swap]{\phi_s} & B \ar{r} \ar[d,equal] & 0 \\
			0 \ar{r} & A \ar{r}{\iota} & G \ar{r}{\pi} & B \ar{r} \arrow[l, bend right=40, "s" swap] & 0,
		\end{tikzcd}
	\end{equation}
where $\phi_s(0,b)=s(b)$ for all $b\in B$. Moreover, with this assigment of topology on $G$, the section $s:B\to G$ is continuous at $0_B$.
\end{comment}

\begin{remark}
	In the proof of \cite[Theorem 2]{Nagao}, the author states that one can choose a section $s:B\to G$ satisfying that $s(-b)=-s(b)$ for all $b\in B$. This assertion is problematic if there exist elements of order 2 in $B$ since \emph{a priori} the element $s(b)$ does not have order 2. Nevertheless, this assumption is not necessary for proving the results in \cite{Nagao}. The important fact is that one should be able to find symmetric neighborhoods of the indentity element, which is always possible.
\end{remark}

\begin{remarks} \label{open nagao}
For $G$ equipped with the topology $\tau_{\s{E},s}$ as in Definition \ref{nagao top}, then a subset $U\subseteq G$ is an open neighborhood of $0_G$ if, and only if, $\iota^{-1}(U)\subseteq A$ and $\pi(U)\subseteq B$ are open neighborhoods of $0_A$ and $0_B$, respectively.
\end{remarks}

\begin{comment}
Recall that first countability is a hereditary property, i.e., it is inherited by (not necessarily closed) subgroups and quotients. Indeed, quotient maps are open. Thus we get the following result.

\begin{proposition} \label{nagao topologizing}
	Let $A \overset{f}{\to} G \overset{g}{\to} B$ be an exact sequence of abelian groups such that $A$ and $B$ are first countable abelian topological groups. Then there exists a topology on $G$ such that $f$ and $g$ are strict continuous.
\end{proposition}

\begin{proof}
	We have the following short exact sequence of abelian groups
	\[ \s{E}:0\to A/\ker f \overset{i}{\to} G \overset{p}{\to} \im g \to 0 \]
	where $A/\ker f$ has the quotient topology from $A$ and $\im g$ has the subspace topology from $B$. The abelian topological groups $A/\ker f$ and $\im g$ are first countable. Thus, choosing a set-theoretic section $s$ of $p$ satisfying $s(0)=0$, we can assign to $G$ the Nagao topology $\tau_{\s{E},s}$ so that $\s{E}$ becomes a topological extension. Now, we can factor $f$ and $g$ as
	\[ f=i\circ\pi_{\ker f} \quad\text{ and }\quad g=\iota_{\im g}\circ p, \]
	where $\pi_{\ker f}:A\to A/\ker f$ is the canonical projection and $\iota_{\im g}:\im g\to B$ is the canonical inclusion. Furthermore, $\pi_{\ker f}$ and $\iota_{\im g}$ are strict continuous. Thus $f$ and $g$ are strict continuous since they are a composition of a quotient map and a topological embedding (see \cite[III, \S 2.8, p. 237, Remark 2]{GTBou}).
\end{proof}
\end{comment}

\begin{remark} \label{choice discrete}
	In relation to Proposition \ref{nagao comparison}, we can add the following.
	\begin{enumerate}[nosep, label=\alph*)]
		\item The topologies $\tau_{\s{E},s}$ and $\tau_{\s{E},s'}$ will be equivalent whenever $B$ has the discrete topology. Consequently, in this case there exists a unique topology on $G$ such that 
		\[ 0 \to A \to G \to B \to 0. \]
		is a topological extension, or equivalently, there exists a unique topology on $G$ such that $A$ is an open subgroup of $G$ (cf. \cite[beginning of \S 3]{DH19}).
		\item Assume that $A$ and $B$ are objects in $\LCA$. If there exist two topologizing sections $s_1$ and $s_2$ of $\pi$ that define different topologies on $G$, then the group $\Ext(B,A)$ of topological extensions is non-trivial (for a detailed treatment of this group, see \cite{FulpGrif}). Indeed, if the topological extensions
		\begin{align*}
			& 0 \to A \to (G,\tau_{\s{E},s_1}) \to B \to 0 \\
			& 0 \to A \to (G,\tau_{\s{E},s_2}) \to B \to 0
		\end{align*}
		are equivalent, then there exists a continuous map $\phi:(G,\tau_{\s{E},s_1})\to(G,\tau_{\s{E},s_2})$ that extends the indentity maps of $A$ and $B$. Thus $\phi$ must be a topological isomorphism (see Corollary 2.2 in \emph{loc. cit.}). But $\tau_{\s{E},s_1}$ and $\tau_{\s{E},s_2}$ differ. Hence the topological extensions defined by $s_1$ and $s_2$ define two different elements in $\Ext(B,A)$. As a consequence of the above, we can ensure that the topology on $G$ is independent of the choice of the section when $A$ is divisible and $B$ is torsion (see \cite[Theorem 1]{Fulp}).
	\end{enumerate}
\end{remark}

As a consequence of Proposition \ref{nagao theorem} the following result follows.

\begin{proposition}
	The property $\s{P}:=\text{``all subgroups of finite index are open''}$ is an extension property for first countable abelian topological groups.
\end{proposition}

\begin{proof}
	Let $\s{E}:0\to A\to G\overset{\pi}{\to} B\to 0$ be a topological extension of first countable abelian topological groups. Assume that $\s{P}$ holds for $A$ and $B$. By Proposition \ref{nagao theorem}, there exists a topologizing section $s$ of $\pi$ that is continuous at $0_B$ and such that $\tau_{\s{E},s}$ agrees with $\tau_G$. Let $U$ be a subgroup of $G$ of finite index. Then $U\cap A$ and $\pi(U)$ are subgroups of $A$ and $B$, respectively, both of finite index and therefore open subgroups. We conclude that $U$ is open in $G$ (see Remark \ref{open nagao}).
\end{proof}

\begin{example}
	Let $G$ be a semiabelian variety defined over a $p$-adic field $k$, i.e., $G$ is an extension of an abelian variety $A$ by a torus $T$. We have a strict exact sequence of first countable abelian topological groups
	\[ 0\to T(k) \to G(k) \to A(k) \to H^1(k,T) \]
(see \cite[Proposition 4.2]{C15}). Moreover, $T(k)$ and $A(k)$ have property $\s{P}$ (see \cite[Chapter I, p. 26]{ADT} and \cite[Chapter I, p. 41, Lemma 3.3]{ADT}, respectively). Consequently, $G(k)$ has property $\s{P}$ since $H^1(k,T)$ is finite. More generally, let $M:=[Y\to G]$ be a Deligne 1-motive defined over $k$, i.e., $M$ is a two-term complex of $fppf$ sheaves concentrated in degrees $-1$ and $0$ with $Y$ equal to the Cartier dual of a torus and $G$ a semiabelian variety. One can similarly prove that the $0$-th hypercohomology group $H^0(k,M)$ has property $\s{P}$ (cf. \cite[Remark 2.4]{HS05}). 
\end{example}

%%%%%%%%%%%%%%%%%%%%%%%%%%%%%%%%%%%%%%%%%%%%%%%%%%%%%%%%%%%%%%%%%%%%%%%%%%%%%%%%%%%%%%%%
%%%%%%%%%%%%%%%%%%%%%%%%%%%%%%%%%%%%%%%%%%%%%%%%%%%%%%%%%%%%%%%%%%%%%%%%%%%%%%%%%%%%%%%%
\subsection*{Continuity of homomorphisms}

Treating with topologies that arises from topologizing section is that allows us to conclude the continuity of some homomorphisms without checking explicitly this condition. The following result illustrates this phenomenon.

\begin{proposition}[Lemma 2.1 \cite{GA24}] \label{P func}
	Let 
	\begin{equation*}
	\begin{tikzcd}
		\s{E}_1: 0 \ar{r} & A_1 \ar{r}{\iota_1} \ar{d}{\alpha} & G_1 \ar{r}{\pi_1} \ar{d}{\gamma} & B_1 \ar{r} \ar{d}{\beta} & 0 \\
		\s{E}_2: 0 \ar{r} & A_2 \ar{r}{\iota_2} & G_2 \ar{r}{\pi_2} & B_2 \ar{r} & 0
	\end{tikzcd}
	\end{equation*}
	be a commutative diagram of abelian groups with exact rows, where $A_i$ and $B_i$ are first countable abelian topological groups $(i=1,2)$. Let $s_i$ be topologizing sections of $\pi_i$. If we endow $G_i$ with the topology $\tau_{\s{E}_i,s_i}$ and the commutativity $\gamma\circ s_1=s_2\circ\beta$ holds, then $\gamma$ is continuous (resp. continuous and strict) if and only if $\alpha$ and $\beta$ are continuous (resp. continuous and strict).
\end{proposition} 

Hereafter, we will study the problem, and its consequences, of obtaining the continuity of homomorphisms from commutative diagrams as above. For convenience, we introduce the following definition. Recall that a map between topological spaces $f:X\to Y$ is continuous at $x\in X$ if and only if for every open neighborhood $V$ of $f(x)$ in $Y$, the set $f^{-1}(V)$ contains an open neighborhood of $x$ in $X$. 

\begin{definition} \label{compatible}
	Let
	\begin{equation*}
		\begin{tikzcd}
			\s{E}_1: 0 \ar{r} & A_1 \ar{r}{\iota_1} \ar{d}{\alpha} & G_1 \ar{r}{\pi_1} \ar{d}{\gamma} & B_1 \ar{r} \ar{d}{\beta} & 0 \\
			\s{E}_2: 0 \ar{r} & A_2 \ar{r}{\iota_2} & G_2 \ar{r}{\pi_2} & B_2 \ar{r} & 0
		\end{tikzcd}
	\end{equation*}
	be a commutative diagram of abelian groups, where $A_i$ and $B_i$ are first countable abelian topological groups ($i=1,2$). Let $s_i$ be a section of $\pi_i$ satisfying $s_i(0)=0$ and define the map
	\begin{align} 
	\begin{split} \label{sigma}
		\sigma_{s_1,s_2}: B_1 &\to A_2 \\
		b &\mapsto \iota_2^{-1}(\gamma\circ s_1(b)-s_2\circ\beta(b)).
	\end{split}
	\end{align}
	We say that $s_1$ and $s_2$ are \emph{compatible} if $\sigma_{s_1,s_2}$ is continuous at $0_{B_1}$.
\end{definition}

While, in general, the condition for two sections to be compatible can be difficult to check, there are some cases in which we can verify it easily.

\begin{examples} \label{exam comp}
	Consider the situation of Definition \ref{compatible}.
	\begin{enumerate}[nosep,label=\alph*.]
		\item If $\gamma\circ s_1=s_2\circ\beta$ holds, then $s_1$ and $s_2$ are compatible. \label{a}
		\item When $B_1$ is discrete, every two sections $s_1$ and $s_2$ are compatible.
		\item When $\beta$ is surjective, for every section $s_1$ of $\pi_1$, there exists a section $s_2$ of $\pi_2$ such that $s_1$ and $s_2$ are compatible. Indeed, given a section $s_1$, we can define $s_2$ as the composition $\gamma\circ s_1\circ\eta$, where $\eta$ is a set-theoretic section of $\beta$ satisfying $\eta(0)=0$, because $\gamma\circ s_1=s_2\circ\beta$ holds and we can appeal to point \ref{a} above.
	\end{enumerate}
\end{examples}

Now, for the convenience of the reader, we recall the following result.

\begin{lemma} \label{sum}
	Let $G$ be a topological group and let $X$ be a topological space. If $f,g:X\to G$ are continuous maps at $x\in X$, then the map $f+g:X\to G$ defined as $(f+g)(x)=f(x)+g(x)$ is continuous at $x$.
\end{lemma}

\begin{proof}
	The map $f+g$ equals the composition 
		\[ X \overset{\Delta}{\to} X\times X \overset{f\times g}{\to} G\times G \overset{+_G}{\to} G, \]
	where $X\times X$ and $G\times G$ have the product topology, $\Delta$ is the diagonal map, $f\times g$ is the product map and $+_G$ is the group operation. Since $f$ and $g$ are continuous at $x\in X$, $f\times g$ is continuous at $(x,x)\in X\times X$. Recall that $\Delta$ is continuous since $\Delta^{-1}(U\times V)=U\cap V$ for every pair of open subsets $U,V\subseteq X$. Hence, as $+_G$ is continuous, we conclude that $f+g=+_G\circ(f\times g)\circ\Delta$ is continuous at $x\in X$.
\end{proof}

The following proposition is a strengthening of Proposition \ref{P func} and tells us that the commutativity condition $\gamma\circ s_1=s_2\circ\beta$ is not really necessary to deduce the continuity of $\gamma$ from those of $\alpha$ and $\beta$. 

\begin{proposition} \label{P3 generalized}
	Let
		\begin{equation}  \label{comp diagram}
			\begin{tikzcd}
				\s{E}_1: 0 \ar{r} & A_1 \ar{r}{\iota_1} \ar{d}{\alpha} & G_1 \ar{r}{\pi_1} \ar{d}{\gamma} & B_1 \ar{r} \ar{d}{\beta} & 0 \\
				\s{E}_2: 0 \ar{r} & A_2 \ar{r}{\iota_2} & G_2 \ar{r}{\pi_2} & B_2 \ar{r} & 0
			\end{tikzcd}
		\end{equation}
	be a commutative diagram of abelian groups, where $A_i$ and $B_i$ are first countable abelian topological groups and $\alpha$ and $\beta$ are continuous. For $i=1,2$, let $s_i$ be a topologizing section of $\pi_i$ and endow $G_i$ with the topology $\tau_{\s{E}_i,s_i}$. If $s_1$ and $s_2$ are compatible \emph{(see Definition \ref{compatible})}, then $\gamma$ is continuous.
\end{proposition}

\begin{proof}
	Using diagram \eqref{nagao diagram}, we have the following commutative diagram of abelian topological groups whose rows are topological extensions
		\begin{equation*} %\label{comp diagram 2}
			\begin{tikzcd}
				0 \ar{r} & A_1 \ar{r}{i_1} \ar[d,equal] & A_1\times B_1 \ar{r}{p_1} \ar{d}{\theta_{s_1}} & B_1 \ar{r} \ar[d,equal] & 0 \\
				0 \ar{r} & A_1 \ar{r}{\iota_1} \ar{d}{\alpha} & G_1 \ar{r}{\pi_1} \ar{d}{\gamma} & B_1 \ar{r} \ar{d}{\beta} & 0 \\
				0 \ar{r} & A_2 \ar{r}{\iota_2} & G_2 \ar{r}{\pi_2} \ar{d}{\theta_{s_2}^{-1}} & B_2 \ar{r} & 0 \\
				0 \ar{r} & A_2 \ar{r}{i_2} \ar[u,equal] & A_2\times B_2 \ar{r}{p_2} & B_2 \ar{r} \ar[u,equal] & 0.
			\end{tikzcd}
		\end{equation*}
	Here the group operation on $A_i\times B_i$ ($i=1,2$) is given by \eqref{operation}. Let us define $\psi$ as the composition of all middle vertical arrows above, i.e.,
		\begin{align} \label{psi}
		\begin{split}
			\psi:A_1\times B &\to A_2\times B_2 \\
			(a,b) &\mapsto (\theta_{s_2}^{-1}\circ\gamma\circ\theta_{s_1})(a,b)
		\end{split}
		\end{align}
	Note that $\gamma$ is continuous if and only if  $\psi$ is continuous at $(0_{A_1},0_{B_1})$ since both maps are group homomorphisms. Now, let us define the maps $\psi_1,\psi_2:A_1\times B_1\to A_2\times B_2$ as follows
		\begin{align} \label{psis}
			\psi_1(a,b):=(\alpha(a),0) \quad \text{and} \quad \psi_2(a,b):=(\sigma_{s_1,s_2}(b),\beta(b)),
		\end{align}
	where $\sigma_{s_1,s_2}$ is the map \eqref{sigma}. We claim that
		\begin{equation} \label{sum psi} 
			\psi(a,b)=\psi_1(a,b) + \psi_2(a,b).
		\end{equation}
	Indeed, by definition of $\theta_{s_i}$ (see \eqref{theta}) and the commutativity of the diagram \eqref{comp diagram}, we have the following equalities
		\begin{align*}
			\phi_{s_2}((\alpha(a),0)+(\sigma_{s_1,s_2}(b),\beta(b))) &= \iota_2(\alpha(a)+\sigma_{s_1,s_2}(b)) + s_2(\beta(b)) \\
			&= \iota_2(\alpha(a)) + (\gamma\circ s_1)(b)-(s_2\circ\beta)(b) + s_2(\beta(b)) \\
			&= \gamma(\iota_1(a)) + \gamma(s_1(b)) \\
			&= (\gamma\circ\phi_{s_1})(a,b).
		\end{align*}
	Let us suppose that $s_1$ and $s_2$ are compatible. We will prove that $\psi$ is continuous. Since $\psi$ is a homomorphism of abelian groups, we only have to check that $\psi$ is continuous at $(0_{A_1},0_{B_1})$. Recall that the set 
	\[ \s{V}_i := \{U_i\times V_i\subset A_i\times B_i:U_i\text{ and } V_i \text{ are open neighborhoods of } 0_{A_i} \text{ and } 0_{B_i} \text{ respectively}\} \]
	is a basis of open neighborhoods of $(0_{A_i},0_{B_i})$ in $A_i\times B_i$.  Since $s_1$ and $s_2$ are compatible, the map $\sigma_{s_1,s_2}$ is continuous at $0_{A_2}$ (see Definition \ref{compatible}). Then, the maps $\psi_1,\psi_2:A_1\times B_1\to A_2\times B_2$ are continuous at $(0_{A_1},0_{B_1})$. Indeed, for every $U_2\times V_2\in\s{V}_2$ we have
	\begin{align*}
		\psi_1^{-1}(U_2\times V_2) &= \alpha^{-1}(U_2)\times B_1, \\
		\psi_2^{-1}(U_2\times V_2) &= A_1\times(\beta^{-1}(V_2)\cap\sigma_{s_1,s_2}^{-1}(U_2)),
	\end{align*}
which are neighborhoods of $(0_{A_1},0_{B_1})$ in $A_1\times B_1$. Hence, by \eqref{sum psi} and Lemma \ref{sum}, $\psi=\psi_1+\psi_2$ is continuous at $(0_{A_1},0_{B_1})$. 
\end{proof}

In contrast to Proposition \ref{P func}, the previous result only allows us to deduce the continuity of $\gamma$ from the continuity of $\alpha$ and $\beta$. However, in the following case we can also ensure the strictness of $\gamma$ from the strictness of $\alpha$ and $\beta$.

\begin{proposition} \label{open fibers}
	Let $s_i$ be a set-theoretic section of $\pi_i$ ($i=1,2$) as in Proposition \ref{P3 generalized}. If the map $\sigma_{s_1,s_2}$ \eqref{sigma} has open fibers, then $\gamma$ is continuous (resp. continuous and strict) if and only if $\alpha$ and $\beta$ are continuous (resp. continuous and strict).
\end{proposition}

\begin{proof}
	Firstly, note that if $\sigma_{s_1,s_2}$ has open fibers, then it is continuous. Thus $s_1$ and $s_2$ are compatible. Hence, by Proposition \ref{P3 generalized}, $\gamma$ is continuous if and only if $\alpha$ and $\beta$ are continuous. Now, let us suppose that $\alpha$ and $\beta$ are continuous and strict. In order to prove that $\gamma$ is strict, we have to prove that $\psi$ in \eqref{psi} is strict. Then we have to check that $\psi(U_1\times V_1)$ is open in $\im\psi$, for every $U_1\times V_1\in\s{V}_1$. Recall that $h_s$ satisfies the equality $h_s(b,0)=h_s(0,b)=0$ for all $b\in B$ and therefore we have the equality 
	\begin{equation} \label{cord zero}
		(a,b)+_s(a',0)=(a+a',b),
	\end{equation}
	for all $a,a'\in A$ and $b\in B$. Then, by \eqref{psis} and \eqref{sum psi}, we have
	\begin{align}
		\im\psi	= \im\psi_1+\im\psi_2 \nonumber &= (\im\alpha\times\{0\})+\left(\bigcup_{b\in B_1}\{\sigma_{s_1,s_2}(b)\}\times\beta(\sigma_{s_1,s_2}^{-1}(\{\sigma_{s_1,s_2}(b)\}))\right) \nonumber \\
		\begin{split} \label{strict 1}
			&= \bigcup_{b\in B_1}(\im\alpha\times\beta(\sigma_{s_1,s_2}^{-1}(\{\sigma_{s_1,s_2}(b)\})))+(\sigma_{s_1,s_2}(b),0) 	
		\end{split} \\
		&\subseteq \bigcup_{b\in B_1}(\im\alpha\times\im\beta)+(\sigma_{s_1,s_2}(b),0), \nonumber
	\end{align}
	here the equality \eqref{strict 1} follows from \eqref{cord zero}. Note that $\sigma_{s_1,s_2}^{-1}(\{\sigma_{s_1,s_2}(b)\})$ is open in $B_1$ since $\sigma_{s_1,s_2}$ has open fibers. Thus $\beta(\sigma_{s_1,s_2}^{-1}(\{\sigma_{s_1,s_2}(b)\}))$ is an open subset of $\im\beta$ since $\beta$ is strict. Let $U_1\times V_1\in\s{V}_1$. Then, similarly to \eqref{strict 1}, we have
	\[ \psi(U_1\times V_1) = \bigcup_{b\in V_1}(\alpha(U_1)\times\beta(\sigma_{s_1,s_2}^{-1}(\{\sigma_{s_1,s_2}(b)\})\cap V_1))+(\sigma_{s_1,s_2}(b),0). \]
	Observe that $\alpha(U_1)$ is open in $\im\alpha$ since $\alpha$ is assumed to be strict. Furthermore, since $\beta$ is strict, $\beta(\sigma_{s_1,s_2}^{-1}(\{\sigma_{s_1,s_2}(b)\})\cap V_1)$ is open in $\im\beta$ and therefore also in $\beta(\sigma_{s_1,s_2}^{-1}(\{\sigma_{s_1,s_2}(b)\}))$. Hence, by the continuity of the group operation, the subset
	\[ (\alpha(U_1)\times\beta(\sigma_{s_1,s_2}^{-1}(\{\sigma_{s_1,s_2}(b)\})\cap V_1))+(\sigma_{s_1,s_2}(b),0) \]
	is open in $\im\psi$ for all $b\in V_1$. Therefore, $\gamma$ is continuous. Conversely, when $\gamma$ is continuous and strict, $\alpha$ is so since $A_i$ is a subspace of $G_i$. Since $\pi_2$ is a quotient map and $\pi$ is surjective, $\beta$ is continuous and strict (see \cite[Lemma 1.8]{GA24}).
\end{proof}

\begin{comment}
The condition on $\sigma_{s_1,s_2}$ to have open fibers seems very restrictive (although it does hold, trivially, if $B_{1}$ is discrete). We do not discuss in this paper the problem of finding other conditions on $\sigma_{s_1,s_2}$ that yield the continuity and strictness of $\gamma$ in \eqref{comp diagram} from those of $\alpha$ and $\beta$.
\end{remark}
\begin{remark}
\end{comment}

In this way, we obtain the following corollary.

\begin{corollary} \label{P3 discrete}
	Let
	\begin{equation*} %\label{functorial diagram}
		\begin{tikzcd}
			\s{E}_1: 0 \ar{r} & A_1 \ar{r}{\iota_1} \ar{d}{\alpha} & G_1 \ar{r}{\pi_1} \ar{d}{\gamma} & B_1 \ar{r} \ar{d}{\beta} & 0 \\
			\s{E}_2: 0 \ar{r} & A_2 \ar{r}{\iota_2} & G_2 \ar{r}{\pi_2} & B_2 \ar{r} & 0
		\end{tikzcd}
	\end{equation*}
	be a commutative diagram of abelian groups, whose rows are topological extensions of first countable abelian topological groups. Assume that $B_1$ is discrete (resp. $A_2$ is indiscrete). Then $\gamma$ is continuous if and only if $\alpha$ is continuous (resp. $\beta$ is continuous). Furthermore, when $B_1$ is discrete, $\gamma$ is continuous and strict if, and only if, $\alpha$ and $\beta$ are continuous and strict.
\end{corollary}

\begin{proof}
	By Propostion \ref{nagao theorem}, there exists topologizing section $s_i$ of $\pi_i$ such that the topology of $G_i$ is $\tau_{\s{E}_i,s_i}$. Now, the map $\sigma_{s_1,s_2}$ is always continuous if $B_1$ is discrete or $A_2$ is indiscrete. Hence, by Proposition \ref{P3 generalized}, $\gamma$ is continuous if $\alpha$ and $\beta$ are so. The converse follows from the fact that $\s{E}_i$ is a topological extension. Finally, when $B_1$ is discrete, the map $\sigma_{s_1,s_2}$ has open fibers. Then, the last assertion follows from Proposition \ref{open fibers}.
\end{proof}

\begin{example}
	Let $G$ and $H$ be first countable abelian topological groups and let $\gamma:G\to H$ be a homomorphism of abelian groups. Let us suppose that $\gamma(\overline{\{0_G\}})\subseteq\overline{\{0_H\}}$. Then $\gamma_\Haus$ is well-defined and, since $\overline{\{0_H\}}$ is indiscrete, $\gamma$ is continuous if and only is $\gamma_\Haus$ is continuous.
\end{example}

%%%%%%%%%%%%%%%%%%%%%%%%%%%%%%%%%%%%%%%%%%%%%%%%%%%%%%%%%%%%%%%%%%%%%%%%%%%%%%%%%%%%%%%%
%%%%%%%%%%%%%%%%%%%%%%%%%%%%%%%%%%%%%%%%%%%%%%%%%%%%%%%%%%%%%%%%%%%%%%%%%%%%%%%%%%%%%%%%
\subsection*{Applications using Pontryagin duality}

Using the relation between Hausdorff compact abelian groups and discrete abelian groups, we obtain the following result.

\begin{proposition} \label{uniqueness separation}
	Let $\s{E}:0\to A\overset{\iota}{\to} G\overset{\pi}{\to} B\to 0$ be an extension of abelian groups, where $A$ and $B$ are first countable locally compact abelian topological groups. Assume that at least one of the following conditions holds:
	\begin{enumerate}[nosep, label=\alph*)]
		\item $A$ is Hausdorff compact; or 
		\item $B$ is discrete and $A$ second countable.
	\end{enumerate}
	Then, for every pair of topologizing sections $s_i:B\to G$ $(i=1,2)$ of $\pi$, $(G,\tau_{\s{E},s_1})^*$ and $(G,\tau_{\s{E},s_2})^*$ are topologically isomorphic. In particular, $(G,\tau_{\s{E},s_1})_{\Haus}$ and $(G,\tau_{\s{E},s_2})_{\Haus}$ are topologically isomorphic.
\end{proposition}

\begin{proof}
	Let $s_1$ and $s_2$ be topologizing sections of $\pi$ $(i=1,2)$. Then
	\[ \s{E}_i:0\to A\overset{\iota}{\to} (G,\tau_{\s{E},s_i}) \overset{\pi}{\to} B\to 0 \]
	is a topological extension of first countable locally compact abelian topological groups. By Lemma \ref{haus exactness}, we obtain the following topological extension
	\[ \s{E}_{s_i}^*:0\to B^* \overset{\pi^*}{\to} (G,\tau_{\s{E},s_i})^* \overset{\iota^*}{\to} A^*\to 0. \]
	Observe that $(G,\tau_{\s{E},s_i})^{***}$ is isomorphic to $(G,\tau_{\s{E},s_i})^*$ by Pontryagin duality since $(G,\tau_{\s{E},s_i})^*$ belongs to $\LCA$. Moreover, recall that the Pontryagin dual of a Hausdorff and compact (resp. discrete) abelian topological group is discrete (resp. Hausdorff and compact). Thus we only have to prove the case when $A$ is Hausdorff and compact. If $A$ is Hausdorff and compact, then $A^*$ is discrete (see Proposition \ref{Pontryagin duality theorem}) and therefore $(G,\tau_{\s{E},s_1})^*$ and $(G,\tau_{\s{E},s_2})^*$ are isomorphic topological groups. The last statement follows from the fact that $G^{**}$ is isomorphic to $G_{\Haus}$ for every locally compact abelian topological group.
\end{proof}

Recalling that the Pontryagin dual of a abelian topological groups is always Hausdorff, we deduce the following corollary from the proposition above.

\begin{corollary}
	Let $\s{E}:0\to A\overset{\iota}{\to} G\overset{\pi}{\to} B\to 0$ be a topological extension of first countable and locally compact abelian topological groups. Assume that $A$ is Hausdorff compact or, $B$ is discrete and $A$ is second countable. Then the compact-open topology on $\Hom_{\cts}(G,\mathbb{T})$ is the unique topology such that
	\[ 0 \to \Hom_{\cts}(B,\mathbb{T}) \overset{\pi^*}{\to} \Hom_{\cts}(G,\mathbb{T}) \overset{\iota^*}{\to} \Hom_{\cts}(A,\mathbb{T}) \to 0 \]
	is a topological extension in $\LCA$ when we equip the left- and right-hand terms with the compact-open topology. Moreover, the compact-open topology on $\Hom_{\cts}(G,\mathbb{T})$ is equivalent to $\tau_{\s{E}^*,s}$ for any topologizing section $s:\Hom_{\cts}(A,\mathbb{T})\to \Hom_{\cts}(G,\mathbb{T})$ of $\iota^*$.
\end{corollary}

Another application of Pontryagin duality using Proposition \ref{nagao theorem} is the following.

\begin{proposition} \label{Five-Lemma nagao}
	Let
	\begin{equation*}
		\begin{tikzcd}
			\s{E}_1: 0 \ar{r} & A_1 \ar{r}{\iota_1} \ar{d}{\alpha} & G_1 \ar{r}{\pi_1} \ar{d}{\gamma} & B_1 \ar{r} \ar{d}{\beta} & 0 \\
			\s{E}_2: 0 \ar{r} & A_2 \ar{r}{\iota_2} & G_2 \ar{r}{\pi_2} & B_2 \ar{r} & 0
		\end{tikzcd}
	\end{equation*}
	be a commutative diagram of abelian groups. Assume that the rows of the above diagram are topological extensions of first countable and locally compact abelian topological groups. Then, in each of the following cases:
	\begin{enumerate}[nosep, label=\alph*)]
		\item $B_1$ is discrete or
		\item $A_2$ is Hausdorff compact and $\s{E}_1$ satisfies some condition of Lemma \ref{haus exactness},
	\end{enumerate}
	the map $\gamma_{\Haus}$ is well-defined and continuous whenever $\alpha$ and $\beta$ are continuous . In particular, if we further assume that $G_2$ is Hausdorff, then $\gamma$ is continuous if, and only if, $\alpha$ and $\beta$ are continuous.
\end{proposition}

\begin{proof}
	By Proposition \ref{nagao theorem}, there exists a topologizing section $s_i$ of $\pi_i$ such that the topology of $G_i$ is $\tau_{\s{E}_i,s_i}$ ($i=1,2$). For the case $a)$, we know that $s_1$ and $s_2$ are compatible since $B_1$ is discrete (see Example \ref{exam comp}.b). Hence, by Corollary \ref{P3 discrete}, we conclude that $\gamma$ is continuous and, therefore, the same holds for $\gamma_{\Haus}$. Assume now that we are in case $b)$. By Lemma \ref{haus exactness}, we obtain the following commutative diagram of abelian groups whose rows are topological extensions
	\begin{equation*}
		\begin{tikzcd}
			\s{E}_1: 0 \ar{r} & B_2^* \ar{r}{\pi_2^*} \ar{d}{\beta^*} & G_2^* \ar{r}{\iota_2^*} \ar{d}{\gamma^*} & A_2^* \ar{r} \ar{d}{\alpha^*} & 0 \\
			\s{E}_2: 0 \ar{r} & B_1^* \ar{r}{\pi_1^*} & G_1^* \ar{r}{\iota_1^*} & A_1^* \ar{r} & 0.
		\end{tikzcd}
	\end{equation*}
	Since $A_2$ is Hausdorff and compact, $A_2^*$ is discrete. Thus, by the previous case, we conclude that $\gamma^*$ is continuous since $\alpha^*$ and $\beta^*$ are continuous. Consequently, $(\gamma^*)^*=\gamma_{\Haus}$ is continuous. The last assertion of the proposition follows from the fact that, when $G_2$ is Hausdorff, the equality $\gamma=\gamma_{\Haus}\circ q_{G_1}$ holds.
\end{proof}

%%%%%%%%%%%%%%%%%%%%%%%%%%%%%%%%%%%%%%%%%%%%%%%%%%%%%%%%%%%%%%%%%%%%%%%%%%%%%%%%%%%%%%%%
%%%%%%%%%%%%%%%%%%%%%%%%%%%%%%%%%%%%%%%%%%%%%%%%%%%%%%%%%%%%%%%%%%%%%%%%%%%%%%%%%%%%%%%%
%%%%%%%%%%%%%%%%%%%%%%%%%%%%%%%%%%%%%%%%%%%%%%%%%%%%%%%%%%%%%%%%%%%%%%%%%%%%%%%%%%%%%%%%
%%%%%%%%%%%%%%%%%%%%%%%%%%%%%%%%%%%%%%%%%%%%%%%%%%%%%%%%%%%%%%%%%%%%%%%%%%%%%%%%%%%%%%%%

\section{Topological Five-Lemma}

We finish this section showing a version of the classical five-lemma in the context of abelian topological groups. 

\begin{theorem}[Topological Five-Lemma] \label{five-lemma theorem}
	Let
	\begin{equation*} 
		\begin{tikzcd}
			A_1 \ar{r}{f_1} \ar{d}{\alpha} & B_1 \ar{r}{g_1} \ar{d}{\beta} & C_1 \ar{r}{h_1} \ar{d}{\gamma} & D_1 \ar{r}{k_1} \ar{d}{\delta} & E_1 \ar{d}{\epsilon} \\
			A_2 \ar{r}{f_2} & B_2 \ar{r}{g_2} & C_2 \ar{r}{h_2} & D_2 \ar{r}{k_2} & E_2
		\end{tikzcd}
	\end{equation*}
	be a commutative diagram of abelian groups, where all rows are strict exact exact sequence of locally compact abelian topological groups. Suppose that
	\begin{enumerate}[nosep,label=\roman*)]
		\item $B_i$ and $D_i$ are first countable,
		\item $\beta$ and $\delta$ are topological isomorphisms,
		\item $\epsilon$ is injective and
		\item $\alpha$ is surjective.
	\end{enumerate}
	Then, in each of the following cases
	\begin{enumerate}[nosep, label=\alph*)]
		\item $D_i$ is discrete; or
		\item $B_i$ is Hausdorff and compact,
	\end{enumerate}
	$\gamma_{\Haus}$ is well-defined and it is a continuous surjective map. Whence $\gamma$ is a continuous bijection when $C_2$ is Hausdorff.
\end{theorem}

\begin{proof}
	By the (classical) Five-Lemma, the map $\gamma$ is an isomorphism of groups. On the one hand, we have the following commutative diagram of abelian topological groups
	\begin{equation*}
		\begin{tikzcd}
			B_1 \ar{r}{\pi_1} \ar{d}{\beta} & B_1/\im f_1 \ar{r} \ar{d}{\beta'} & 0\\
			B_2 \ar{r}{\pi_2} & B_2/\im f_2 \ar{r} & 0,
		\end{tikzcd}
	\end{equation*}
	where $B_i/\im f_i$ is equipped with the quotient topology. Then the map $\beta'$ is continuous since $\beta$ is so. On the other hand, endowing $\im h_i$ with the subspace topology from $D_i$, the restriction $\delta':\im h_1\to \im h_2$ of $\delta$ is continuous since $\delta$ is so. Thus, we obtain the following commutative diagram of abelian topological groups whose rows are strict exact
	\begin{equation*}
		\begin{tikzcd}
			\s{E}_1: 0 \ar{r} & B_1/\im f_1 \ar{r}{g_1'} \ar{d}{\beta'} & C_1 \ar{r}{h_1'} \ar{d}{\gamma} & \im h_1 \ar{r} \ar{d}{\delta'} & 0 \\
			\s{E}_2: 0 \ar{r} & B_2/\im f_2 \ar{r}{g_2'} & C_2 \ar{r}{h_2'} & \im h_2 \ar{r} & 0,
		\end{tikzcd}
	\end{equation*}
	where $\beta'$ and $\delta'$ are continuous bijections. Note that $B_i/\im f_i$ and $\im h_i$ are first countable locally compact topological groups. Further, in case $b)$, the quotient $B_i/\im f_i$ is Hausdorff compact since $\im f_i$ is closed. Then, by Proposition \ref{Five-Lemma nagao}, $\gamma_{\Haus}$ is well-defined and it is continuous. Since $\gamma$ is bijective, $\gamma_{\Haus}$ is also surjective and, therefore, $\gamma_{\Haus}$ is a continuous surjective map. Finally, if $C_2$ is Hausdorff, then $\gamma=\gamma_{\Haus}\circ q_{C_1}$. Hence, $\gamma$ is continuous bijection.
\end{proof}

\begin{remark}
	\noindent
	\begin{enumerate}[nosep, label=\alph*.]
		\item Note that $\alpha$ and $\epsilon$ are not assumed to be continuous.
		\item We may relax condition $i)$ by assuming that $\beta$ and $\delta$ are continuous isomorphisms of abelian groups (that is, their inverses need not be continuous). Furthermore, when $C_1$ is second countable, checking the continuity of $\gamma_{\Haus}$ is sufficient for concluding that $\gamma_{\Haus}$ is a quotient map (see Remark \ref{strict}.d). In particular, when $C_2$ is Hausdorff and $C_1$ is second countable, the map $\gamma$ is an topological isomorphism.
	\end{enumerate}
\end{remark}

\begin{comment}
The following remarks justify our Hausdorff assumptions in Theorem \ref{Five-Lemma}.

\begin{remark} \label{duality}
	Let $k$ be $p$-adic field and let $\phi:X\to\Spec k$ be a proper and geometrically integral variety over $k$. Let $\s{C}$ be the complex $\sHom_{D(k_{\sm})}(\tau_{\leq 1}R\phi_*\G_{m,X},\G_{m,k})$. There exists a natural way to topologize the cohomology groups of the complex $\s{C}$ but, unfortunately, the $0$-th cohomology group of $\s{C}$ is often non-Hausdorff. Nevertheless, we would like to prove that the map
	\[ H^0(k,\s{C})^\wedge \to \Br_1(X)^*,\]
induced by a Yoneda pairing is a topological isomorphism. In order to do this, we have to ensure that the map
	\[ H^0(k,\s{C}) \to \Br_1(X)^*,\]
	is continuous. At this point, we need to apply the results obtained in Section \ref{nagao section} above. The details will appear elsewhere.
\end{remark}
\end{comment}

\end{document}